\documentclass[11pt]{article} 

\usepackage{preamble} 

\usepackage{tipa} 

\usepackage{linguex} 

\usepackage{graphicx} 

\usepackage{float} 

\usepackage{amsmath, amssymb} 

\usepackage{amsfonts} 

\usepackage{tikz-cd} 

\usepackage{mathabx} 

\usepackage{enumitem} 

\usepackage{capt-of} 

\title{Knocking Down Boxes: \\The FMP for $\mathbf{K}\oplus \Box^{m+k} p\to\Box^m p$
} 

\author{Søren Brinck Knudstorp \\ 
    \small ILLC and Philosophy, University of Amsterdam \\
    \small Amsterdam, The Netherlands \\
    \small \texttt{s.b.knudstorp@uva.nl} \\ 
}

\begin{document}

\setheaderinfo{Knocking Down Boxes}{S. B. Knudstorp} 

\maketitle 
\thispagestyle{ACtitlepageFH} 
\startmainpages 

\begin{abstract}
    It is a long-standing open problem whether modal logics of the form $\mathbf{K}\oplus \Box^np\to\Box^mp$ for $n>m>1$ have the finite model property (FMP). We solve this by showing that any modal logic axiomatized by formulas of the form $\Box^np\to\Box^mp$ where $n>m>1$ has the FMP.
\end{abstract}
\section{Introduction}

Among the more basic results in modal logic are that the logics $\mathbf{K}\oplus\Box p\to p$ and $\mathbf{K}\oplus\Box\Box p\to \Box p$ have the finite model property (FMP), and likewise for the logics axiomatized by the converse axioms, $\mathbf{K}\oplus p\to \Box p$ and $\mathbf{K}\oplus\Box p\to \Box\Box p$. Yet for the innocuous-looking analogues 
\[
    \mathbf{K}\oplus\Box\Box\Box p\to \Box\Box p\qquad \text{and}\qquad \mathbf{K}\oplus\Box\Box p\to \Box\Box\Box p,
\]
it is unknown whether they too enjoy the FMP. In fact, they are the simplest instances of the following long-open problem:
\begin{align}
    \text{Do the logics of the form }\mathbf{K}\oplus \Box^np\to\Box^mp,\text{for $n,m>1$, have the FMP?}\tag{P}\label{problem:fmp}    
\end{align}
This question appears in  Problem 11.2 in \textcite{ChagrovZ97:ml}, who attribute it to Segerberg (see their 11.8 Notes). Earlier, \textcite{Gabbary72:JPL} had established the FMP in the cases where $n=1$ or $m=1$, and around the same time, \textcite{Zakharyaschev1997} proved that all extensions of $\mathbf{K4}$ by \textit{modal reduction principles}---formulas of the form $\mathsf{M}p\to\mathsf{N}p$ for $\mathsf{M,N}$ sequences of $\Diamond$s and $\Box$s---have the FMP. \textcite{Zakharyaschev1997} also remarked that ``[...] the situation with FMP of extensions of $\mathbf{K}$ by modal reduction principles, even by axioms of the form $\Box^np \to \Box^mp$ still remains unclear. I think at present this is one of the major challenges in completeness theory''. A decade later, \textcite{WolterZakharyaschev2007} echoed this, reiterating the question~\eqref{problem:fmp} as Problem 6 and calling it ``[p]erhaps one of the most intriguing open problems in Modal Logic''.

Only recently has there been partial breakthrough. \textcite{LyonOstropolski-Nalewaja24:lics} used techniques from database theory, in particular (disjunctive) existential rules and (disjunctive) chase, to prove that if $\Phi$ is a finite set of formulas of the form $\Box^np\to \Box^mp$ where $n>m>1$, then $\mathbf{K}\oplus \Phi$ is decidable.\footnote{Their proof contains some errors, but in personal communication, Piotr Ostropolski-Nalewaja has let me know that they are aware of this, believe them to be non-critical, and are working on a corrected extended version.}

Motivated by this, we solve half of problem~\eqref{problem:fmp}, showing that if $\Phi$ is a set of formulas of the form $\Box^np\to \Box^mp$ where $n>m>1$, then $\mathbf{K}\oplus \Phi$ has the FMP. As a corollary we recover the result of \textcite{LyonOstropolski-Nalewaja24:lics}: if $\Phi$ is finite, then $\mathbf{K}\oplus \Phi$ is decidable. In contrast to their work, we prove the FMP and employ a different approach that combines a representation lemma with a filtration lemma.

We proceed as follows. Section~\ref{sec:Setup} sets out the preliminaries, the proof strategy, and the main results. Section~\ref{sec:RepandFil} contains the proofs: Subsection~\ref{subsec:Rep} proves a representation lemma, and Subsection~\ref{subsec:Fil} a filtration lemma; together, these yield the finite model property.

\section{Setup and Results}\label{sec:Setup}
We begin with some foundational definitions; for further background, readers may consult standard textbooks on modal logic, e.g., \textcite{bluebook}. 

Our \textit{language} is the basic modal language, generated by the grammar 
\[
    \varphi\mathrel{::=} \bot\hspace{0.1cm}|\hspace{0.1cm}p \hspace{0.1cm}| \hspace{0.1cm} \neg\varphi\hspace{0.1cm}|\hspace{0.1cm}\varphi\lor\varphi\hspace{0.1cm}|\hspace{0.1cm}\Diamond\varphi,
\]
where $p\in \mathsf{Prop}$ are propositional variables and $\bot$ is the falsum constant. `$\land$', `$\to$', and `$\Box$' are defined in the usual way. 

\textit{Frames} are pairs $(W, R)$ where $R\subseteq W\times W$; \textit{pointed frames} $(W, R,r)$ are frames $(W, R)$ with $r\in W$; and \textit{models} $(W, R, V)$ are frames $(W, R)$ with a \textit{valuation} $V:\mathsf{Prop}\to \mathcal{P}(W)$. Given a frame $(W, R)$, the $k$-step relations $R^k\subseteq W\times W$ are defined recursively by 
\[
    R^0\mathrel{:=}\Delta=\{(w,w)\mid w\in W\},\qquad R^{k+1}\mathrel{:=}R^k\circ R,
\]
and we write $R^k[w]\mathrel{:=}\{v\in W\mid wR^kv\}$ for the set of $k$-step successors of $w\in W$. A pointed frame $(W, R,r)$ is \textit{rooted} (at $r$) if $\bigcup_{k\in\omega} R^k[r]=W$.

We let $\mathbf{K}$ denote the least normal modal logic, and recall that formulas of the form $\Box^np\to\Box^mp$ are equivalent modulo $\mathbf{K}$ to $\Diamond^mp\to\Diamond^np$; for instance,
\[
    \mathbf{K}\oplus\Diamond\Diamond p\to \Diamond\Diamond\Diamond p=\mathbf{K}\oplus\Box\Box\Box p\to \Box\Box p,
\]
where, for a formula or set of formulas $\Gamma$, $\mathbf{K}\oplus\Gamma$ is the least normal modal logic containing $\Gamma$. 
Henceforth, we adopt the diamond formulation, and reserve the notation $\Phi$ for an arbitrary set of formulas of the form $\Diamond^mp\to \Diamond^n p$ for $n>m$. 
By Sahlqvist, $\mathbf{K}\oplus\Phi$ is sound and complete for the class of frames $(W,R)$ satisfying the following first-order condition, for each $\Diamond^mp\to \Diamond^n p\in \Phi$:
\[
    \forall x,y(xR^my\to xR^ny).
\]
We will show that it is complete for the class of \textit{finite} such frames. Our proof proceeds by combining a representation lemma with a filtration lemma, for which we first need a few further definitions. 

\begin{definition}
    Let $(W,R)$ be a frame with root $r$. The \textit{depth} of a world $w\in W$ w.r.t. $r$ is the length of a shortest path from $r$ to $w$. We denote this by 
    \[
        \text{d}_r(w)\mathrel{:=}\min\{k\in \omega\mid rR^kw\},
    \] 
   although we will often suppress the subscript and simply write $\text{d}(w)$. Note that $\text{d}_r(r)=0$.
\end{definition}

\begin{definition}
    The \textit{modal depth} of a formula $\varphi$, written $\text{md}(\varphi)$, is the maximal number of nested modalities in $\varphi$. That is, 
    \begin{align*}
        \text{md}(p) &= 0 \\
        \text{md}(\bot) &= 0 \\
        \text{md}(\neg\varphi) &= \text{md}(\varphi) \\
        \text{md}(\varphi \lor \psi) &= \max\{\text{md}(\varphi), \text{md}(\psi)\} \\
        \text{md}(\Diamond\varphi) &= \text{md}(\varphi)+1.
    \end{align*}
\end{definition}
\begin{definition}
    Let $(W, R, r)$ and $(W', R', r')$ be pointed frames. We say that a map $f:(W, R, r)\to (W', R',r')$ is a (pointed) \textit{p-morphism} if 
    \begin{enumerate}[leftmargin=50pt]
        \item [\textnormal{(pointed)}] $f(r)=r'$,
        \item [\textnormal{(forth)}] if $xRy$, then $f(x)R'f(y)$,
        \item [\textnormal{(back)}] if $f(x)R'y'$, then there is $y\in W$ s.t. $f(y)=y'$ and $xRy$.
    \end{enumerate}
    
\end{definition}
\begin{definition}
    Let $\mathfrak{M}=(W,R,V)$ and $\mathfrak{M}'=(W',R',V')$ be models. For $x\in W, x'\in W'$, we say that $x$ and $x'$ are \textit{bisimilar}, and write $x\sim x'$, if there is a relation $Z\subseteq W\times W'$ such that $(x,x')\in Z$ and for all $(w,w')\in Z$:
    \begin{enumerate}[leftmargin=50pt]
        \item [\textnormal{($a$)}] $\text{for all $p\in \mathsf{Prop}$: }w\in V(p)\Leftrightarrow w'\in V'(p)$,
        \item [\textnormal{($b$)}] \text{if $wRv$, then there is $v'\in W'$ with $w'R'v'$ and $vZv'$},
        \item [\textnormal{($c$)}] \text{if $w'R'v'$, then there is $v\in W$ with $wRv$ and $vZv'$.}    \end{enumerate}
\end{definition}
The following is well-known and readily proven.
\begin{fact}\label{fact:bis}
     Let $f:(W',R',r')\to (W, R, r)$ be a p-morphism. For any model $(W,R,V)$ over $(W,R)$, setting $V'(p)\mathrel{:=}f^{-1}[V(p)]$ yields a model $(W', R', V')$ such that for all $y,y'\in W'$,
     \[
        y\sim f(y),
     \]
     and
     \[
        f(y')=f(y)\quad \text{implies}\quad y'\sim y.
     \]
     Moreover, for any models $(W,R,V)$, $(W',R',V')$ and points $x\in W, x'\in W'$, if $x\sim x'$, then $x$ and $x'$ satisfy the same formulas.
\end{fact}
With this, let us outline the proof architecture. It consists of two lemmas: a representation lemma and a filtration lemma. Together, they produce a finite $\mathbf{K}\oplus\Phi$ model from any given $\mathbf{K}\oplus\Phi$ model. We state these lemmas next, omitting their proofs for now, and show how they combine to establish the finite model property for $\mathbf{K}\oplus\Phi$. (Recall that, throughout, $\Phi$ contains only formulas $\Diamond^m p\to\Diamond^n p$ with $n>m$.)

\begin{lemma}[Representation]\label{lm:depthcondition}
    Let $\varphi$ be a formula. If $\varphi$ is satisfiable in a $\mathbf{K}\oplus\Phi $ frame, then it is satisfied at a root $r\in W$ of a model $(W,R, V)$ such that:
    \begin{align}
        \forall x,y\big(xRy\to \exists y'[xRy',\; \textnormal{d}_r(y')=\textnormal{d}_r(x)+1,\; y'\sim y]\big),\label{condition:depthbis}
    \end{align}
    and for each $\Diamond^mp\to \Diamond^np\in \Phi$,
    \begin{align}
        \forall x,z \;\;\exists w_1, \hdots, w_{n-1} 
        \big(&xR^mz\to[xRw_1R\cdots Rw_{n-1}Rz,\; \textnormal{d}_r(w_{j})=\textnormal{d}_r(x)+j]\big).\label{condition:depth}
    \end{align}
\end{lemma}

\begin{lemma}[Filtration]\label{lm:filtration}
    Let $\varphi$ be a formula. If $\varphi$ is satisfied at a root $r\in W$ of a model $(W,R,V)$ fulfilling \eqref{condition:depthbis} and \eqref{condition:depth} for each $\Diamond^mp\to \Diamond^np\in \Phi$, then $\varphi$ is satisfiable in a finite $\mathbf{K}\oplus\Phi $ frame.
\end{lemma}

\begin{theorem}
    $\mathbf{K}\oplus\Phi $ enjoys the finite model property.
\end{theorem}
\begin{proof}
    Suppose $\varphi$ is satisfiable in a $\mathbf{K}\oplus\Phi $ frame. Applying the preceding lemmas in sequence, we get that $\varphi$ is satisfiable in a finite $\mathbf{K}\oplus\Phi $ frame.
\end{proof}
\begin{corollary}
    For $\Phi$ finite, $\mathbf{K}\oplus\Phi $ is decidable.
\end{corollary}
\begin{proof}
    Finitely axiomatizable normal modal logics with the finite model property are decidable.
\end{proof}

\section{Proof of the Representation and Filtration Lemmas}\label{sec:RepandFil}
\subsection{Round One: The Representation Lemma}\label{subsec:Rep}
It remains to prove the Representation Lemma (Lemma~\ref{lm:depthcondition}) and the Filtration Lemma (Lemma~\ref{lm:filtration}). We begin with the former. Conceptually, the lemma is straightforward: it formalizes the idea that in a $\mathbf{K}\oplus\Diamond^mp\to\Diamond^np$ frame, whenever $xR^mz$, there exist $w_1,\hdots, w_{n-1}$ s.t. $xRw_1R\cdots Rw_{n-1}Rz$; and if $w_1,\hdots, w_{n-1}$ fail to have the depth required by condition~\eqref{condition:depth}, we may add `duplicates', which we will be denoting $u_1, \hdots, u_{n-1}$, and extend $R$ to a relation $R'\supseteq R$ by defining $xR'u_{1}R'\cdots R'u_{n-1}R'z$ so that the only paths to each $u_{j}$ pass through $x$, whence the $u_{j}$s will have the required depth. Importantly, although $R'\supseteq R$, the procedure preserves the depth of all points already present before the extension, and hence can be iterated without undoing earlier steps. This takes care of condition~\eqref{condition:depth}; an analogous idea handles condition~\eqref{condition:depthbis}. 

To make this precise, we use a step-by-step argument that applies two sublemmas: one to enforce \eqref{condition:depth}, and another to enforce \eqref{condition:depthbis}.\footnote{The method is standard and follows a setup similar to those in \textcite{Knudstorp2023:jpl,Knudstorp2023:tbillc}. See also \textcite{bluebook} for an introduction to the step-by-step method.} We begin with the sublemma for condition~\eqref{condition:depth}. For clarity, note that the subscripts $_i$ and $_{i+1}$ in $(W_i, R_i, r_i)$, $f_i$, and so on, used in the lemma statement are \textit{not} variables but fixed labels, chosen for convenience in our later proof of Lemma~\ref{lm:depthcondition}.
\begin{convention}
Since $\mathbf{K}\oplus \Phi$ frames are closed under point-generated subframes, which are rooted by definition, we henceforth assume that pointed frames $(W,R,r)$ are rooted at $r$.
\end{convention}
\begin{lemma}[sublemma for~\eqref{condition:depth}]\label{lm:conddepth}
    Let $f_{i}: (W_{i}, R_{i}, r_i) \to  (W, R,r)$ be a p-morphism, for $(W_{i}, R_{i})$ a \textnormal{Kripke frame} and $(W, R)$ a \textnormal{$\mathbf{K}\oplus\Phi$ frame}. Then for every $\Diamond^mp\to\Diamond^np\in \Phi $, if $xR_{i}^mz$, then \textnormal{(a)} there are $w_1,\hdots, w_{n-1}\in W$ and $v_{1},\hdots, v_{n-1}\in W_{i}$ such that 
    \[
        f_{i}(x)Rw_1R\cdots Rw_{n-1}Rf_{i}(z),\qquad xR_{i}v_1R_i\cdots R_{i}v_{n-1},\qquad f_{i}(v_j)=w_j;
    \]
    and \textnormal{(b)} for $u_{1},\hdots, u_{n-1}$ distinct points not in $W_i$, the map $f_{i+1}:(W_{i+1}, R_{i+1}, r_i)\to (W,R,r)$ given by
\[
\begin{aligned}
  W_i \ni y &\mapsto f_i(y) \\
  u_j &\mapsto w_j,
\end{aligned}
\]      
where $W_{{i+1}} \mathrel{:=} W_{i} \cup \{u_{1},\hdots, u_{n-1}\}$ and
\[
R_{{i+1}} \mathrel{:=} R_{i} \cup \{(x,u_1), \hdots, (u_{j}, u_{j+1}), \hdots, (u_{n-1}, z)\} \cup \bigcup_{1\leq j<n}\{(u_{j},c)\mid v_{j}R_{i} c\},
\]
defines a p-morphism such that 
\[
     \textnormal{d}_{i+1}(u_j)=\textnormal{d}_{i+1}(x)+j\quad \text{and}\quad \textnormal{d}_{i+1}(y)=\textnormal{d}_{i}(y) \text{ for all $y\in W_{i}$}.
\]
Here, for all $y\in W_{i+1}$, $\textnormal{d}_{i+1}(y)$ denotes the length of a shortest path from $r_{i}$ to $y$ in $(W_{i+1}, R_{i+1})$; and $\textnormal{d}_{i}(y)$ denotes the length of a shortest path from $r_{i}$ to $y$ in $(W_{i}, R_{i})$.
\end{lemma}

\begin{proof}
    Let $\Diamond^mp\to\Diamond^np\in \Phi$ be arbitrary and suppose $xR_{i}y_1R_i\cdots R_i y_{m-1}R_{i}z$. By the forth clause, we get that $f_{i}(x)Rf_{i}(y_1)R\cdots R f_{i}(y_{m-1})Rf_{i}(z)$. So as $(W, R)$ is a $\mathbf{K}\oplus\Phi $ frame, there must be $w_1,\hdots, w_{n-1}\in W$ s.t. $f_{i}(x)Rw_1R\cdots Rw_{n-1}Rf_{i}(z)$. Applying the back clause $n-1$ times, we additionally find $v_1,\hdots,  v_{n-1}\in W_{i}$ s.t. $xR_{i}v_1R_i\cdots R_{i}v_{n-1}$ and $f_{i}(v_{j})=w_j$, for all $j\in\{1,\hdots, n-1\}$. This proves (a).

    Consequently, defining $W_{i+1}, R_{i+1}$, and $f_{i+1}$ by adding fresh points $u_{1},\hdots, u_{n-1}$ as in the lemma statement, it only remains to show that $f_{i+1}$ is a p-morphism such that
    \[
        \textnormal{d}_{i+1}(u_j)=\textnormal{d}_{i+1}(x)+j\text{, for all $j\in \{1,\hdots, n-1\}$,}\qquad \text{and}\qquad \textnormal{d}_{i+1}(y)=\textnormal{d}_{i}(y) \text{, for all $y\in W_{i}$}.
    \]
    We begin with showing that $f_{i+1}$ is a p-morphism. Observe that $f_{i+1}(r_i)=f_i(r_i)=r$, since $f_i$ is a pointed p-morphism. For the forth condition, assume $aR_{i+1}b$. We prove that $f_{i+1}(a)Rf_{i+1}(b)$ by cases. 
    \begin{itemize}
        \item If $(a,b)\in R_{i}$, then forth for $f_{i}$ gives us $f_{i+1}(a)=f_{i}(a)Rf_{i}(b)=f_{i+1}(b)$.
        \item If $(a,b)\in \{(x,u_1), \hdots,(u_{j}, u_{j+1}), \hdots,(u_{n-1}, z)\}$, then $f_{i+1}(a)Rf_{i+1}(b)$ follows because $f_{i}(x)Rw_1R\cdots Rw_{n-1}Rf_{i}(z)$.
        \item If $(a,b)=(u_{j}, c)$ for $c\in W_{i}$ s.t. $v_{j}R_{i}c$, then the claim follows by forth for $f_{i}$ because $f_{i+1}(u_{j})=w_j=f_{i}(v_{j})Rf_{i}(c)=f_{i+1}(c)$.
    \end{itemize}
    
For the back condition, suppose $f_{i+1}(a)Rb'$. If $a\in W_{i}$, then $f_{i}(a)=f_{i+1}(a)Rb'$, so by the back condition for $f_{i}$, there is $b\in W_{i}$ s.t. $aR_{i}b$ and $f_{i}(b)=b'$, hence also $aR_{i+1}b$ and $f_{i+1}(b)=f_{i}(b)=b'$. If $a\notin W_{i}$, so $a=u_{j}$, then $f_{i}(v_{j})=w_j=f_{i+1}(u_{j})Rb'$, so by the back condition for $f_{i}$, there is $b\in W_{i}$ s.t. $f_{i}(b)=b'$ and $v_{j}R_{i}b$. But then $u_{j}R_{i+1}b$ and $f_{i+1}(b)=f_{i}(b)=b'$. Thus, $f_{i+1}$ meets the back condition, and we may therefore conclude that it is a p-morphism.

Lastly, we have to show that $\textnormal{d}_{i+1}(u_j)=\textnormal{d}_{i+1}(x)+j$ and $\textnormal{d}_{i+1}(y)=\textnormal{d}_{i}(y)$ for all $y\in W_{i}$. To begin, observe that $(W_{i+1}, R_{i+1})$ is rooted at $r_{i}$, hence $\textnormal{d}_{i+1}(y)$ is well-defined for all $y\in W_{i+1}$. Now to see that $\textnormal{d}_{i+1}(u_j)=\textnormal{d}_{i+1}(x)+j$, note that all paths from $r_{i}$ to $u_{1}$ must start with an $R_{i}$-path to $x$ and end with $(x,u_{1})\in R_{i+1}$; similarly, paths to $u_{j+1}$ must start with an $R_{i+1}$-path to $u_{j}$ and end with $(u_{j},u_{j+1})\in R_{i+1}$. We thus obtain 
\[
    \textnormal{d}_{i+1}(u_j)=\textnormal{d}_{i+1}(x)+j,\quad \text{along with}\quad \textnormal{d}_{i+1}(x)=\textnormal{d}_{i}(x).
\]
We still need to verify that $\textnormal{d}_{i+1}(y)=\textnormal{d}_{i}(y)$ for all $y\in W_{i}$. For convenience, write $u_0\mathrel{:=}x$ and $u_n\mathrel{:=}z$, and observe that for each $j\in \{1,\hdots, n-1\}$,
\begin{align*}
    &\textnormal{d}_{i}(v_{j})\leq \textnormal{d}_{i}(x)+j =\textnormal{d}_{i+1}(u_{j}),\\
    &R_{i+1}[u_{j}]\setminus\{u_{j+1}\}\subseteq R_{i}[v_{j}].
\end{align*}
This suffices because then for any $R_{i+1}$-path (starting at $r_i$):
\begin{itemize}
    \item if it contains $(u_{n-1},u_n)\in R_{i+1}$, then there is a not longer $R_{i}$-path to $u_n=z$ as $xR_{i}y_1R_i\cdots R_i y_{m-1}R_{i}z$ implies that $\textnormal{d}_{i}(z)\leq \textnormal{d}_{i}(x)+m=\textnormal{d}_{i+1}(u_m)\leq \textnormal{d}_{i+1}(u_{n-1})$;
    \item if it contains $(u_{j},u_{j+1})\in R_{i+1}$ for $j<n-1$ followed by $(u_{j+1},a)\in R_{i+1}$ for $a\neq u_{j+2}$, then there is a not longer $R_{i}$-path to $a$ as $\textnormal{d}_{i}(v_{j+1})\leq\textnormal{d}_{i+1}(u_{j+1})$ and $R_{i+1}[u_{j+1}]\setminus\{u_{j+2}\}\subseteq R_{i}[v_{j+1}]$ implies $v_{j+1}R_{i}a$.
\end{itemize} 
Thus $\textnormal{d}_{i+1}(y)=\textnormal{d}_i(y)$ for all $y\in W_i$, which concludes the proof of the lemma.
\end{proof}
Next, we state and prove the sublemma used for enforcing condition~\eqref{condition:depthbis}.

\begin{lemma}[sublemma for~\eqref{condition:depthbis}]\label{lm:conddepthbis}
    Let $f_{i}: (W_{i}, R_{i}, r_i) \to  (W, R,r)$ be a p-morphism, for $(W_{i}, R_{i})$ a \textnormal{Kripke frame} and $(W, R)$ a \textnormal{$\mathbf{K}\oplus\Phi$ frame}. If $wR_{i}v$, then for $u\notin W_i$, the map $f_{i+1}:(W_{i+1}, R_{i+1},r_i)\to (W,R,r)$ given by
\[
\begin{aligned}
  W_i \ni y &\mapsto f_i(y) \\
  u &\mapsto f_i(v),
\end{aligned}
\]      
where
\[
W_{{i+1}} \mathrel{:=} W_{i} \cup \{u\},\quad R_{{i+1}} \mathrel{:=} R_{i} \cup \{(w,u)\} \cup \{(u,c)\mid vR_{i} c\},
\]
defines a p-morphism such that 
\[
     \textnormal{d}_{i+1}(u)=\textnormal{d}_{i+1}(w)+1\quad \text{and}\quad \textnormal{d}_{i+1}(y)=\textnormal{d}_{i}(y) \text{ for all $y\in W_{i}$.}
\]
\end{lemma}

\begin{proof}
    Define $W_{i+1}, R_{i+1}$, and $f_{i+1}$ by adding a point $u\notin W_i$ as in the lemma statement. Both that $f_{i+1}$ is a p-morphism and that
\[
     \textnormal{d}_{i+1}(u)=\textnormal{d}_{i+1}(w)+1\quad \text{and}\quad \textnormal{d}_{i+1}(y)=\textnormal{d}_{i}(y) \text{ for all $y\in W_{i}$}
\]
follow by simplified versions of the corresponding arguments from the preceding lemma.
\end{proof}

\begin{remark}\label{rm:depthpreservation}
Note, in particular, that both sublemmas preserve the depth of already existing points: when passing from $R_i$ to $R_{i+1}\supseteq R_i$, we still have $\textnormal{d}_{i+1}(y)=\textnormal{d}_i(y)$ for every $y\in W_i$.
\end{remark}

With these sublemmas at hand, and keeping this observation in mind, 
we can now prove the Representation Lemma (Lemma~\ref{lm:depthcondition}).

\begin{proof}[Proof of the Representation Lemma~\ref{lm:depthcondition}]
    If $\varphi$ is satisfiable in a $\mathbf{K}\oplus\Phi $ frame, then as $\mathbf{K}\oplus\Phi $ frames are closed under point-generated subframes, $\varphi$ is satisfiable at a root of a $\mathbf{K}\oplus\Phi$ frame. 
    
    It therefore suffices to show that for every $\mathbf{K}\oplus\Phi$ frame $(W,R,r)$, there are a frame $(W', R', r')$ and a pointed p-morphism $f:(W',R',r') \to (W,R,r)$ satisfying \eqref{condition:depth} and
    \begin{align}
        \forall x,y\big(xR'y\to \exists y'[xR'y',\; \textnormal{d}_{r'}(y')=\textnormal{d}_{r'}(x)+1,\; f(y')= f(y)]\big).\label{condition:depthpmorphism}
    \end{align}
    This suffices by Fact~\ref{fact:bis}, because for any model $\mathfrak{M}=(W,R,V)$ over $(W,R)$, defining $V'(p)\mathrel{:=}f^{-1}[V(p)]$ yields a model $\mathfrak{M}'=(W',R', V')$ over $(W',R')$ such that $r'\sim f(r')=r$, and $f(y')=f(y)$ implies $y'\sim y$.
    
    Accordingly, let $(W,R,r)$ be a rooted $\mathbf{K}\oplus\Phi $ frame. We may assume without loss of generality that $|W|\leq \aleph_0$, since the frame conditions for $\mathbf{K}\oplus\Phi $ are first-order and $\Phi $ is countable, so it follows by Löwenheim--Skolem through the standard translation that $\mathbf{K}\oplus\Phi $ has the countable model property.\footnote{Alternatively, instead of assuming $W$ to be countable, one can apply transfinite recursion in what follows.}

    We construct $(W', R', r')$ and $f$ step by step, applying the preceding lemmas iteratively. 
    
    First, we set $(W_0, R_0,r_0)\mathrel{:=}(W,R,r)$, and let $f_0\mathrel{:=}\mathrm{id}:(W,R,r)\to (W,R,r)$ be the identity map. 
    
    Then, we fix a countably infinite set $X$ disjoint from $W$, and enumerate all tuples of
    \[
        (W\cup X)^2\cup \bigcup_{\Diamond ^mp\to \Diamond^n p\in \Phi }(W\cup X)^{2}\times \{\Diamond^mp\to \Diamond^np\},
    \]
    which is countable as $\Phi , W, X$ all are countable. $(W_{i+1}, R_{i+1}, r_{i+1})$ and $f_{i+1}$ are then constructed from $(W_{i}, R_{i}, r_i)$ and $f_{i}$ by applying the 
    relevant preceding sublemma to the least tuple 
    \[
        \overline{x}\in (W\cup X)^2\cup \bigcup_{\Diamond ^mp\to \Diamond^n p\in \Phi }(W\cup X)^{2}\times \{\Diamond^mp\to \Diamond^np\}
    \]
    of our enumeration constituting a \textit{defect} for $(W_{i}, R_{i}, r_i), f_i$. (If there are no defects, set $(W_{i+1}, R_{i+1}, r_{i+1})\mathrel{:=}(W_{i}, R_{i}, r_i)$ and $f_{i+1}\mathrel{:=}f_i$.)
    
    There are two types of defects. 
    First, a tuple of the form $\bar{x}=(x,y)\in (W\cup X)^2$ constitutes a defect for $(W_{i}, R_{i}, r_i), f_i$ if
    \[
        xR_iy\quad \text{yet there is no $y'\in W_i$ s.t. $xR_iy',\; \textnormal{d}_i(y')=\textnormal{d}_i(x)+1,\; f_i(y')=f_i(y)$,} 
    \]
    and in this case Lemma~\ref{lm:conddepthbis} is applied.

    Second, a tuple of the form $\bar{x}=(x,z, \Diamond^mp\to\Diamond^np)$ constitutes a defect for $(W_{i}, R_{i}, r_i), f_i$ if
    \begin{align*}
        xR_i^mz\quad &\text{yet there are no $w_1,\hdots,  w_{n-1}\in W_i$}\text{ s.t. $xR_iw_1R_i\cdots R_iw_{n-1}R_iz\textnormal{ and }\textnormal{d}_i(w_j)=\textnormal{d}_i(x)+j$,}
    \end{align*}
    and in this case Lemma~\ref{lm:conddepth} is applied.
        
The added points---namely, $\{u\}=W_{i+1}\setminus W_{i}$ when applying Lemma~\ref{lm:conddepthbis}, and $\{u_1, \hdots, u_{n-1}\}=W_{i+1}\setminus W_{i}$ when applying Lemma~\ref{lm:conddepth}---are drawn from $X\setminus W_i$, which is possible as $X\setminus W_i$ is infinite for all $i\in \omega$, since we only draw finitely many points in each application of a sublemma.

Letting 
\[
    W'\mathrel{:=}\bigcup_{i\in \omega} W_i, \quad R'\mathrel{:=}\bigcup_{i\in \omega} R_{i}, \quad r'\mathrel{:=} r, \quad f\mathrel{:=}\bigcup_{i\in\omega}f_i,
\]
we have to show that
\begin{itemize}
    \item $f:(W', R')\to (W,R)$ is a p-morphism with $f(r')=r$; and
    \item $(W',R')$ is rooted at $r'$ and satisfies conditions \eqref{condition:depth} and \eqref{condition:depthpmorphism}.
\end{itemize}
That $f$ is well-defined as a function follows from all $f_i$ being functions and that, by definition, $f_{i+1}{\upharpoonright} W_i=f_i$; and that $f(r')=r$ is a consequence of $r'=r$ and $f_0=\mathrm{id}$. It satisfies the forth clause because each $f_i$ does, and, likewise, it satisfies the back clause because each $f_i$ does. Thus, we conclude that $f$ is a p-morphism.

Next, that $(W',R')$ is rooted at $r'=r$ follows from each $(W_i, R_i)$ being rooted at $r_i=r$. 

Lastly, to see that it satisfies the conditions~\eqref{condition:depth} and~\eqref{condition:depthpmorphism}, first observe that for all $w\in W'=\bigcup_{i\in \omega} W_i$, there is a least $i\in\omega$ s.t. $w\in W_i$ and, as highlighted in Remark~\ref{rm:depthpreservation}, for all $j\geq i$, $\textnormal{d}_j(w)=\textnormal{d}_i(w)$; hence $\textnormal{d}'(w)=\textnormal{d}_i(w)$. 

Now for condition \eqref{condition:depth}, let $\Diamond^mp\to \Diamond^np\in \Phi$ be arbitrary and suppose $xR^{\prime^m}
z$. Then there is a least $i\in\omega$ s.t. $\{x,z\}\subseteq W_i$ and $xR_i^mz$. We also have $\textnormal{d}'(x)=\textnormal{d}_l(x)$ for all $l\geq i$. If there are $w_1, \hdots, w_{n-1}\in W_i$ s.t. $xR_iw_1R_i\cdots R_iw_{n-1}R_iz\textnormal{ and }\textnormal{d}_i(w_j)=\textnormal{d}_i(x)+j$, then we have $xR'w_1R'\cdots R'w_{n-1}R'z\textnormal{ and }\textnormal{d}'(w_j)=\textnormal{d}'(x)+j$, as required. If not, then as $(x,z)\in W_i^{2}\subseteq (W\cup X)^{2}$, the tuple $(x,z,\Diamond^mp\to\Diamond^np)$ has been assigned some number $k\in\omega$ in our enumeration. So by construction, at some stage $l\leq i+k+1$, Lemma~\ref{lm:conddepth} will have been applied to $(x,z,\Diamond^mp\to\Diamond^np)$ [or to another tuple whose resolution happens to fix this defect too],\footnote{Here we tacitly use that a defect resolved at stage $i$ remains resolved for all $j\geq i$, because $R_j\supseteq R_i$, $f_{j}{\upharpoonright} W_i=f_i$, and $\textnormal{d}_j(w)=\textnormal{d}_i(w)$ for all $w\in W_i$; otherwise the construction could, for example, oscillate between the first two tuples of the enumeration.} and points $u_1,\hdots, u_{n-1}$ will have been added so that $xR_{l}u_1R_{l}\cdots R_{l}u_{n-1}R_{l}z\textnormal{ and }\textnormal{d}_{l}(u_j)=\textnormal{d}_{l}(x)+j$. Consequently, also in this case, the condition holds, as then $xR'u_1R'\cdots R'u_{n-1}R'z\textnormal{ and }\textnormal{d}'(u_j)=\textnormal{d}'(x)+j$.

This shows that condition \eqref{condition:depth} holds. The proof of condition \eqref{condition:depthpmorphism} is similar, but for completeness, we include it here as well. Accordingly, suppose $xR'y$. Then there is a least $i\in\omega$ s.t. $\{x,y\}\subseteq W_i$, and $\textnormal{d}'(x)=\textnormal{d}_l(x)$ for all $l\geq i$. If there is $y'\in W_i$ s.t. $xR_iy', \textnormal{d}_i(y')=\textnormal{d}_i(x)+1,\text{ and }f_i(y')=f_i(y)$, then we have  $xR'y', \textnormal{d}'(y')=\textnormal{d}'(x)+1,\text{ and }f(y')=f(y)$, as required. If not, then as $(x,y)\in W_i^{2}\subseteq (W\cup X)^{2}$, the tuple has been assigned some number $k\in\omega$ in our enumeration. So by construction, at some stage $l\leq i+k+1$, Lemma~\ref{lm:conddepthbis} will have been applied to $(x,y)$, and a point $u$ will have been added so that $xR_{l}u, \textnormal{d}_{l}(u)=\textnormal{d}_{l}(x)+1, \text{ and }f_{l}(u)=f_{l}(y)$. Consequently, also in this case, the condition holds, as then $xR'u, \textnormal{d}'(u)=\textnormal{d}'(x)+1, \text{ and }f(u)=f(y)$, which thereby completes the proof of the lemma.
\end{proof}
\subsection{Round Two: The Filtration Lemma}\label{subsec:Fil}
Next and last, we prove the Filtration Lemma (Lemma~\ref{lm:filtration}). For this, we need one final definition.
\begin{definition}
    Let $\mathfrak{M}=(W,R,V)$ and $\mathfrak{M}'=(W',R',V')$ be models. For a set of propositional variables $\mathbf{X}\subseteq \mathsf{Prop}$, we define relations ${\sim^{\mathbf{X}}_k}\subseteq W\times W'$ recursively as follows for $(w,w')\in W\times W'$:
    \begin{align*}
        w\sim^{\mathbf{X}}_0w' \qquad &\text{iff}\qquad \hspace{.2cm}\text{for all $p\in \mathbf{X}$: }w\in V(p)\Leftrightarrow w'\in V'(p)\\
        w\sim^{\mathbf{X}}_{k+1}w' \qquad &\text{iff}\qquad \begin{array}[t]{ll}
        (a) & w\sim^{\mathbf{X}}_0 w', \\[3pt]
        (b) & \text{if $wRv$, then there is $v'\in W'$ with $w'R'v'$ and $v\sim^{\mathbf{X}}_k v'$,} \\[3pt]
        (c) & \text{if $w'R'v'$, then there is $v\in W$ with $wRv$ and $v\sim^{\mathbf{X}}_k v'$.}
      \end{array}
    \end{align*}
    In case $w\sim^{\mathbf{X}}_kw'$, we say that $w$ and $w'$ are \textit{$k$-bisimilar (w.r.t. $\mathbf{X}$)}. When $\mathbf{X}$ is clear from context, we omit the superscript and write $w\sim_kw'$.
\end{definition}
Before proceeding to the proof of the lemma, we record some well-known and readily proven facts regarding $k$-bisimilarity.
\begin{fact}\label{fact:kbis}
     Let $\mathfrak{M}=(W,R,V)$ and $\mathfrak{M}'=(W',R',V')$ be models, and $\mathbf{X}\subseteq \mathsf{Prop}$ a set of propositional variables. Then the following hold for all $k\in\omega$:
     \begin{itemize}
         \item If $w\sim^{\mathbf{X}}_kw'$, then for all formulas $\varphi$  such that $\mathsf{Prop}(\varphi)\subseteq \mathbf{X}$ and $\textnormal{md}(\varphi)\leq k$, 
     \[
        \mathfrak{M},w\Vdash\varphi \qquad\text{iff}\qquad \mathfrak{M}', w'\Vdash \varphi.
     \]
        \item ${\sim}\subseteq{\sim^{\mathbf{X}}_{k+1}}\subseteq {\sim^{\mathbf{X}}_k}$.
        \item If $\mathfrak{M}=\mathfrak{M}'$, then $\sim^{\mathbf{X}}_k$ is an equivalence relation on $W$.
        \item If $\mathfrak{M}=\mathfrak{M}'$ and $\mathbf{X}$ is finite, then the set of $\sim^{\mathbf{X}}_k$-equivalence classes is finite.
     \end{itemize}
\end{fact}
\begin{proof}[Proof of the Filtration Lemma~\ref{lm:filtration}]
    Suppose $\varphi$ is satisfied at a root $r\in W$ of a model $\mathfrak{M}=(W,R, V)$ satisfying \eqref{condition:depthbis} and \eqref{condition:depth} for each $\Diamond^mp\to \Diamond^np\in \Phi$. Letting $\textnormal{d}(x)\mathrel{:=}\textnormal{d}_{r}(x)$ and $k\mathrel{:=}\text{md}(\varphi)$, we extract a finite $\varphi$-satisfying $\mathbf{K}\oplus\Phi $ model through a particular kind of filtration.
    First, we define a binary relation $\equiv$ on $W_{\leq k}\mathrel{:=}\{x\in W\mid \textnormal{d}(x)\leq k\}$ as follows:
    \[
        x\equiv y \qquad \text{iff} \qquad \textnormal{d}(x)=\textnormal{d}(y) \text{ and } x\sim_{k-\textnormal{d}(x)}y,
    \]
    where we have suppressed the superscript $^{\mathsf{Prop}(\varphi)}$ and written $x\sim_{k-\textnormal{d}(x)}y$ rather than \text{$x\sim_{k-\textnormal{d}(x)}^{\mathsf{Prop}(\varphi)}y$}.

    Since $x\in W_{\leq k}$ means $\textnormal{d}(x)\leq k$, we have $k-\textnormal{d}(x)\in\omega$ for all $x\in W_{\leq k}$, so $\sim_{k-\textnormal{d}(x)}$ is well-defined. One then easily checks that $\equiv$ is an equivalence relation on $W_{\leq k}$ using that the relations $\sim_{l}$ are equivalence relations, cf. Fact~\ref{fact:kbis}. Then, by the last bullet point of Fact~\ref{fact:kbis} and the fact that the set of propositional variables occurring in $\varphi$ is finite, i.e. $|\mathsf{Prop}(\varphi)|<\aleph_0$, we have that the set of $\equiv$-equivalence classes 
    \[
        W^f\mathrel{:=}\{|x| \mid x\in W_{\leq k}\}
    \]
    is finite, where $|x|$ denotes the $\equiv$-equivalence class of $x$.

    We then define a relation $R^f\subseteq W^f\times W^f$ by setting $|x|R^f|y|$ \text{iff} $\textnormal{d}(x)=k$ or 
    \[
        \textnormal{d}(x)<k,\quad \textnormal{d}(y)\leq \textnormal{d}(x)+1, \quad \text{and}\quad \exists y'(xRy'\text{ and } y'\sim_{k-\textnormal{d}(x)-1}y).
    \]
    Take note that we do \textit{not} require that $y'\in |y|$ (and that $\textnormal{d}(x)<k$ and $xRy'$ jointly imply $y'\in W_{\leq k}$). 
    
    Let us show that $R^f$ is well-defined, in that it doesn't depend on the representatives of the equivalence classes. To see that it doesn't depend on the representative for $|x|$, observe that $x\equiv x'$ implies $\textnormal{d}(x)=\textnormal{d}(x')$ and $x\sim_{k-\textnormal{d}(x)}x'$, so if $xRy'$ and $y'\sim_{k-\textnormal{d}(x)-1}y$, then $x\sim_{k-\textnormal{d}(x)}x'$ implies that there is $y''$ s.t. $x'Ry''$ and $y''\sim_{k-\textnormal{d}(x)-1}y'$, hence also $y''\sim_{k-\textnormal{d}(x)-1}y$. And to see that it doesn't depend on the representative for $|y|$, assume $y\equiv y''$. Then $y\sim_{k-\textnormal{d}(y)}y''$, so $\textnormal{d}(y)\leq \textnormal{d}(x)+1$ implies $y\sim_{k-\textnormal{d}(x)-1}y''$, cf. Fact~\ref{fact:kbis}. Hence $y'\sim_{k-\textnormal{d}(x)-1}y$ would imply $y'\sim_{k-\textnormal{d}(x)-1}y''$.
    
    Thus, $(W^f, R^f)$ is a well-defined finite frame. To show that it is a $\mathbf{K}\oplus\Phi $ frame, we will use that the map $w\mapsto |w|$ is an $R$-homomorphism: if $xRy$ then $|x|R^f|y|$, for all $x,y\in W_{\leq k}$. Indeed, if $xRy$, then clearly $\textnormal{d}(y)\leq \textnormal{d}(x)+1$ and $y\sim_{k-\textnormal{d}(x)-1}y$, so we have $|x|R^f|y|$.
    \\\\
    Now let $\Diamond^mp\to \Diamond^np\in \Phi$ be arbitrary and suppose $|x|R^f|y_1|R^f\cdots R^f |y_{m-1}| R^f|z|$. We are to find $|w_1|,\hdots, |w_{n-1}|$ such that $|x|R^f|w_1|R^f\cdots R^f|w_{n-1}|R^f|z|$. First, we observe that 
    \[
        \textnormal{d}(y_i)\leq \textnormal{d}(x)+i\qquad \text{and}\qquad \textnormal{d}(z)\leq \textnormal{d}(x)+m,
    \]
    since, whenever $|w|R^f|v|$, we have $\textnormal{d}(v)\leq \textnormal{d}(w)+1$: this is clear if $\textnormal{d}(w)<k$, and follows from $\textnormal{d}(v)\leq k$ if $\textnormal{d}(w)=k$. Then, we continue by cases on $\textnormal{d}(x)$.  

    \textbf{Case 1.} If $\textnormal{d}(x)\leq k-m$, then we claim that there are $y_i'\in W$, for $i\in \{1,\hdots, m\}$, such that
    \[
        xRy_1'Ry_2'R\cdots Ry_m',\qquad \textnormal{d}(y_i')\leq \textnormal{d}(x)+i, \qquad y_i'\sim_{k-\textnormal{d}(x)-i}y_i,
    \]
    where we for convenience have denoted $y_m\mathrel{:=}z$. We argue by induction on $i\in \{1,\hdots, m\}$. 

    For the base case, since $|x|R^f|y_1|$ and $\textnormal{d}(x)\leq k-m<k$, there is $y_1'$ such that $xRy_1'$ and $y_1'\sim_{k-\textnormal{d}(x)-1}y_1$. The former implies $\textnormal{d}(y_1')\leq \textnormal{d}(x)+1$.
    
    For the inductive step, since $|y_j|R^f|y_{j+1}|$ and $\textnormal{d}(y_j)\leq \textnormal{d}(x)+j<k$, we likewise find $y_{j+1}''$ such that $y_jRy_{j+1}''$ and $y_{j+1}''\sim_{k-\textnormal{d}(y_j)-1}y_{j+1}$. Since $\textnormal{d}(y_j)\leq \textnormal{d}(x)+j$, Fact~\ref{fact:kbis} gives us $y_{j+1}''\sim_{k-\textnormal{d}(x)-j-1}y_{j+1}$. Further, by the induction hypothesis, we have $y_j'\sim_{k-\textnormal{d}(x)-j}y_j$, so since $y_jRy_{j+1}''$, there is $y_{j+1}'$ s.t. $y_j'Ry_{j+1}'$ and $y_{j+1}'\sim_{k-\textnormal{d}(x)-j-1}y_{j+1}''$. From the former we get 
    \[
        \textnormal{d}(y_{j+1}')\leq \textnormal{d}(y_{j}')+1\overset{\text{IH}}{\leq}\textnormal{d}(x)+j+1,
    \]
    and from the latter we get $y_{j+1}'\sim_{k-\textnormal{d}(x)-j-1}y_{j+1}$, which completes the induction and establishes the claim. 
    
    By assumption, $(W,R)$ satisfies \eqref{condition:depth}, so $xRy_1'Ry_2'R\cdots Ry_m'$ entails that there are $w_1, \hdots, w_{n-1}\in W$ s.t. 
    \[
        xRw_1R\cdots Rw_{n-1}Ry_m'\qquad \text{and}\qquad  \textnormal{d}(w_{j})=\textnormal{d}(x)+j.
    \]
    Using this, we further split the case according to whether $\textnormal{d}(x)\leq k-n$.
    
    \noindent
\begin{minipage}[t]{0.56\textwidth} 
    \vspace{0pt}
    \setlength{\abovedisplayskip}{0pt} 
\setlength{\abovedisplayshortskip}{0pt}
    \textbf{Case 1a.} If $\textnormal{d}(x)> k-n$, then there is $w_j$ with $\textnormal{d}(w_j)=k$ for $1\leq j\leq n-1$. So as $w\mapsto |w|$ is a homomorphism, we get a $j$-path $|x|R^f|w_1|R^f\cdots R^f|w_{j}|$, and using that $|w_{j}|R^f|w_{j}|$ and $|w_{j}|R^f|z|$ as $\textnormal{d}(w_j)=k$, we can extend it to an $n$-path ending at $|z|$.

    \textbf{Case 1b.} If $\textnormal{d}(x)\leq k-n$, then 
    \[
        \textnormal{d}(w_{j})=\textnormal{d}(x)+j\leq k-n+j<k,
    \]
    so $w_j\in W_{\leq k}$ for all $j\leq n-1$, hence 
    \[
        |x|R^f|w_1|R^f\cdots R^f|w_{n-1}|.
    \]
    Thus, it will suffice to show that $|w_{n-1}|R^f|z|$. Since $\textnormal{d}(w_{n-1})<k$, we have to show that
\end{minipage}%
  \hspace{0.05\textwidth}
\begin{minipage}[t]{0.39\textwidth} 
\centering
\begin{tikzpicture}[>=stealth, node distance=1.5cm, baseline=(current bounding box.north)]
  
  \node (a2) at (6,0) {$x$};
  \node (c2) at (8,2) {$y_1$};
  \node (c3) at (10,4) {$y_2=z$};
  
  \node (d2) at (6,2) {$y_1'$};
  \node (d3) at (8,4) {$y_2''$};      
  \node (d4) at (6,4) {$y_2'$};    

  \node (y1) at (5,1.33) {$w_1$};
  \node (y2) at (5,2.66) {$w_2$};

  \node[draw, align=center] at (8.8,0.7) {\small $\textnormal{d}(z)\leq \textnormal{d}(w_2)+1$ \\ $y_2'\sim_{k-\textnormal{d}(w_2)-1}z$};

  \node at (7,2) {$\sim$};
    \node at (9,4) {$\sim$};
    \node at (7,4) {$\sim$};

  \draw[->] (a2) -- (d2);
    \draw[->] (c2) -- (d3);
    \draw[->, dashed] (d2) -- (d4);
    );

    \draw[->, dotted] (a2) .. controls (5.2,0.66) .. (y1);
    \draw[->, dotted] (y1) -- (y2);
    \draw[->, dotted] (y2) .. controls (5.2,3.33) .. (d4);
\end{tikzpicture}

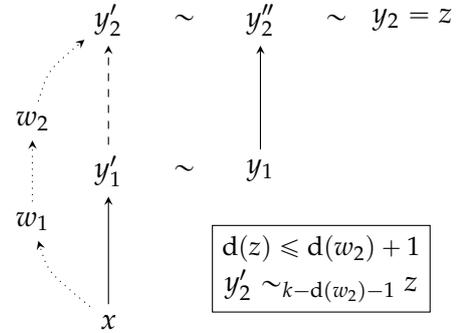
\captionof{figure}{Illustration of the key part of the argument, when $\textnormal{d}(x)\leq k-n$, for $m=2,n=3$, so $|x|R^f|y_1|R^f|z|\Rightarrow \exists |w_1|, |w_2|\,{:}\; |x|R^f|w_1|R^f|w_2|R^f|z|$.}\label{fig:relation}
\end{minipage}
\[
        \textnormal{d}(z)\leq \textnormal{d}(w_{n-1})+1 \qquad \text{and}\qquad \exists z'(w_{n-1}Rz'\text{ and } z'\sim_{k-\textnormal{d}(w_{n-1})-1}z).
\]
    The former holds because
\[
        \textnormal{d}(z)\leq \textnormal{d}(x)+m\leq \textnormal{d}(x)+n=\textnormal{d}(w_{n-1})+1.
\]
    The latter is witnessed by $y_m'$, in view of Fact~\ref{fact:kbis} combined with 
\[
    y_m'\sim_{k-\textnormal{d}(x)-m}y_m=z \qquad \text{and}\qquad \textnormal{d}(x)+m\leq \textnormal{d}(x)+n =\textnormal{d}(w_{n-1})+1.
\]
This completes the case where $\textnormal{d}(x)\leq k-m$.

\textbf{Case 2.} If $\textnormal{d}(x)>k-m$, then the subcase $\textnormal{d}(x)=k$ is immediate, since in that case $|x|R^f|x|$ and $|x|R^f|z|$. Accordingly, assume further that $\textnormal{d}(x)<k$, i.e. $\{0,\hdots, k-\textnormal{d}(x)-1\}\neq\varnothing$. 

We denote $y_0\mathrel{:=}x$ and use that for all $i\in \{0,\hdots, k-\textnormal{d}(x)-1\}$, there is $y_{i+1}'$ s.t. 
\[
    y_iRy_{i+1}'    \qquad \text{and} \qquad y_{i+1}'\sim_{k-\textnormal{d}(x)-i-1}y_{i+1}.
\]
This follows because $\textnormal{d}(y_i)\leq \textnormal{d}(x)+i$ and, for $i\leq k-\textnormal{d}(x)-1$, we have $\textnormal{d}(x)+i\leq k-1<k$. Thus, for such $i$, $|y_i|R^f|y_{i+1}|$ entails that there is $y_{i+1}'$ s.t. $y_iRy_{i+1}'$ and $y_{i+1}'\sim_{k-\textnormal{d}(y_i)-1}y_{i+1}$. And the latter implies $y_{i+1}'\sim_{k-\textnormal{d}(x)-i-1}y_{i+1}$, since $\textnormal{d}(y_i)\leq \textnormal{d}(x)+i$.

Strengthening this, we claim that for all $i\in \{0,\hdots, k-\textnormal{d}(x)-1\}$, there is $y_{i+1}''$ s.t. 
\[
        xRy_1''Ry_2''R\cdots Ry_{k-\textnormal{d}(x)}'',\qquad \textnormal{d}(y_{i+1}'')= \textnormal{d}(x)+i+1, \qquad y_{i+1}''\sim_{k-\textnormal{d}(x)-i-1}y_{i+1}.
    \]
We prove this by induction on $i\in \{0,\hdots, k-\textnormal{d}(x)-1\}$. 

For the base case, since $x=y_0Ry_1'$, we get by condition~\eqref{condition:depthbis} that there is $y_1''$ s.t. $xRy_1'', \textnormal{d}(y_1'')=\textnormal{d}(x)+1,$ and $y_1''\sim y_1'$. Since $y_{1}'\sim_{k-\textnormal{d}(x)-1}y_{1}$, it follows that $y_1'' \sim_{k-\textnormal{d}(x)-1}y_1$.

For the inductive case, since $y_iRy_{i+1}'$ and, by IH, $y_{i}''\sim_{k-\textnormal{d}(x)-i}y_{i}$, there is $y_{i+1}'''$ s.t. $y_{i}''Ry_{i+1}'''$ and $y_{i+1}'''\sim_{k-\textnormal{d}(x)-i-1}y_{i+1}'$. So by condition~\eqref{condition:depthbis}, there is $y_{i+1}''$ s.t.
\[
    y_i''Ry_{i+1}'', \qquad \textnormal{d}(y_{i+1}'')=\textnormal{d}(y_i'')+1\overset{\text{IH}}{=}\textnormal{d}(x)+i+1,\qquad y_{i+1}''\sim y_{i+1}'''.
\]
Thus, as 
\[
    y_{i+1}'''\sim_{k-\textnormal{d}(x)-i-1}y_{i+1}'\sim_{k-\textnormal{d}(x)-i-1}y_{i+1},
\]
we also have $y_{i+1}'' \sim_{k-\textnormal{d}(x)-i-1}y_{i+1}$, as desired.

With the induction completed, it now follows that $|x|R^f|y_1''|R^f|y_2''|R^f\cdots R^f|y_{k-\textnormal{d}(x)}''|$, because $w\mapsto |w|$ is a homomorphism. This path can be extended to an $n$-path ending at $|z|$, as we have $|y_{k-\textnormal{d}(x)}''|R^f|y_{k-\textnormal{d}(x)}''|$ and $|y_{k-\textnormal{d}(x)}''|R^f|z|$ because $\textnormal{d}(y_{k-\textnormal{d}(x)}'')=\textnormal{d}(x)+k-\textnormal{d}(x)=k$.

This completes the case where $\textnormal{d}(x)>k-m$, and we may therefore conclude that $(W^f, R^f)$ is a $\mathbf{K}\oplus\Phi$ frame. 
\\\\    
Finally, we let $\mathfrak{M}^f\mathrel{:=}(W^f,R^f,V^f)$, where $V^f$ is the valuation on $(W^f, R^f)$ defined by setting
\[
    V^f(p)\mathrel{:=}\{|x|\mid x\in V(p)\},
\]
for all $p\in \mathsf{Prop}(\varphi)$. Note that this is well-defined. We proceed to show that for all $x\in W_{\leq k}$ and $\psi\in \operatorname{sub}(\varphi)$ with $\textnormal{md}(\psi)\leq k-\textnormal{d}(x)$: 
    \[
        \mathfrak{M},x\Vdash \psi\qquad \text{iff}\qquad \mathfrak{M}^f, |x|\Vdash \psi.
    \]
    As $\mathfrak{M},r\Vdash \varphi, \textnormal{d}(r)=0, \textnormal{md}(\varphi)=k$, and $\varphi\in \operatorname{sub}(\varphi)$, this would imply $\mathfrak{M}^f, |r|\Vdash \varphi$, thereby completing the proof.

    We prove this claim by an induction on the modal depth of $\psi\in \operatorname{sub}(\varphi)$. For the base case, we prove the claim holds for all $\psi\in \operatorname{sub}(\varphi)$ with $\textnormal{md}(\psi)=0$. Since we, by definition, have 
    \[
        \mathfrak{M},x\Vdash p\qquad \text{iff}\qquad \mathfrak{M}^f, |x|\Vdash p
    \]
    for all $x\in W_{\leq k}$ and all propositional variables $p\in \mathsf{Prop}(\varphi)$, an easy induction establishes this biimplication for all Boolean combinations of propositional variables from $\mathsf{Prop}(\varphi)$, hence for all formulas $\psi\in \operatorname{sub}(\varphi)$ with $\textnormal{md}(\psi)=0$.

    For the inductive step, suppose for some $i\in\{0,\hdots, k-1\}$ that the claim holds for all $\psi\in \operatorname{sub}(\varphi)$ with $\textnormal{md}(\psi)\leq i$. We show it holds for all $\psi\in \operatorname{sub}(\varphi)$ with $\textnormal{md}(\psi)=i+1$. The proof is by structural induction on $\psi$, where, once again, the Boolean cases are straightforward. So suppose $\psi=\Diamond\chi$. Then $\textnormal{md}(\psi)=i+1$ implies $\textnormal{md}(\chi)=i$; and $\psi\in \operatorname{sub}(\varphi)$ implies $\chi\in \operatorname{sub}(\varphi)$. Let $x\in W_{\leq k}$ with $\textnormal{md}(\psi)\leq k-\textnormal{d}(x)$ be arbitrary, and note that $\textnormal{d}(x)\leq k-i-1$. 

    If $\mathfrak{M},x\Vdash \Diamond\chi$, there must be $y\in W$ such that $xRy$ and $\mathfrak{M},y\Vdash \chi$. Since $xRy$ and $\textnormal{d}(x)\leq k-i-1$, we get that $\textnormal{d}(y)\leq k-i$, whence $\textnormal{md}(\chi)=i\leq k-\textnormal{d}(y)$ and $y\in W_{\leq k}$. Hence we can apply the induction hypothesis, yielding
    \[
        \mathfrak{M}^f, |y|\Vdash \chi
    \]
    and thus witnessing 
    \[
        \mathfrak{M}^f, |x|\Vdash \Diamond\chi, 
    \]
    as $w\mapsto |w|$ is a homomorphism, so $xRy$ implies $|x|R^f|y|$, for $x,y\in W_{\leq k}$.

    Conversely, if $\mathfrak{M}^f,|x|\Vdash \Diamond\chi$, then there is $|y|$ such that $|x|R^f|y|$ and $\mathfrak{M}^f,|y|\Vdash \chi$. Observe that since $\textnormal{d}(x)\leq k-i-1$, we have $\textnormal{d}(x)<k$, hence $|x|R^f|y|$ entails that $\textnormal{d}(y)\leq \textnormal{d}(x)+1$ \text{and} there is $y'\in W$ with $xRy'\text{ and } y'\sim_{k-\textnormal{d}(x)-1}y$. So we have 
    \[
        \textnormal{md}(\chi)=i\leq k-\textnormal{d}(x)-1\leq k-\textnormal{d}(y),
    \]
    hence by the induction hypothesis,
    \[
        \mathfrak{M}, y\Vdash \chi.
    \]
    Lastly, because $y'\sim_{k-\textnormal{d}(x)-1}y$ and $\textnormal{md}(\chi)\leq k-\textnormal{d}(x)-1$, Fact~\ref{fact:kbis} yields
    \[
        \mathfrak{M}, y'\Vdash \chi,
    \]
    whence 
        \[
        \mathfrak{M}, x\Vdash \Diamond\chi,
    \]
    as required.
\end{proof}


\paragraph{Acknowledgements.} I thank an anonymous referee for helpful suggestions on the exposition.

\paragraph{Funding.} The research was supported by the Nothing is Logical (NihiL) project (NWO OC 406.21.CTW.023).





\printbibliography\label{biblio} 

\end{document}